\documentclass[12pt]{amsart}

\usepackage{psfrag}
\usepackage{color}

\usepackage{graphicx,graphics}
\usepackage{fullpage,amssymb,amsfonts,amsmath,amstext,amsthm,amscd,enumerate,verbatim,tikz}
\usepackage[T1]{fontenc}
\usetikzlibrary{matrix,arrows}

\usepackage[all]{xy} \usepackage{url}

\begin{document}
\newcommand{\note}[1]{\marginpar{\tiny #1}}
\newtheorem{theorem}{Theorem}[section]
\newtheorem{result}[theorem]{Result}
\newtheorem{fact}[theorem]{Fact}
\newtheorem{conjecture}[theorem]{Conjecture}
\newtheorem{definition}[theorem]{Definition}
\newtheorem{lemma}[theorem]{Lemma}
\newtheorem{proposition}[theorem]{Proposition}
\newtheorem{remark}[theorem]{Remark}
\newtheorem{corollary}[theorem]{Corollary}
\newtheorem{facts}[theorem]{Facts}
\newtheorem{question}[theorem]{Question}
\newtheorem{props}[theorem]{Properties}

\newtheorem{ex}[theorem]{Example}

\newcommand{\notes} {\noindent \textbf{Notes.  }}
\renewcommand{\note} {\noindent \textbf{Note.  }}
\newcommand{\defn} {\noindent \textbf{Definition.  }}
\newcommand{\defns} {\noindent \textbf{Definitions.  }}
\newcommand{\x}{{\bf x}}
\newcommand{\z}{{\bf z}}
\newcommand{\B}{{\bf b}}
\newcommand{\V}{{\bf v}}
\newcommand{\T}{\mathcal{T}}
\newcommand{\Z}{\mathbb{Z}}
\newcommand{\Hp}{\mathbb{H}}
\newcommand{\D}{\mathbb{D}}
\newcommand{\R}{\mathbb{R}}
\newcommand{\N}{\mathbb{N}}
\renewcommand{\B}{\mathbb{B}}
\newcommand{\C}{\mathbb{C}}
\newcommand{\dt}{{\mathrm{det }\;}}
 \newcommand{\adj}{{\mathrm{adj}\;}}
 \newcommand{\0}{{\bf O}}
 \newcommand{\av}{\arrowvert}
 \newcommand{\zbar}{\overline{z}}
 \newcommand{\htt}{\widetilde{h}}
\newcommand{\ty}{\mathcal{T}}
\renewcommand\Re{\operatorname{Re}}
\renewcommand\Im{\operatorname{Im}}
\newcommand{\diam}{\operatorname{diam}}
\newcommand{\dist}{\text{dist}}
\newcommand{\ds}{\displaystyle}
\numberwithin{equation}{section}
\newcommand{\cN}{\mathcal{N}}
\renewcommand{\theenumi}{(\roman{enumi})}
\renewcommand{\labelenumi}{\theenumi}
\newcommand{\inte}{\operatorname{int}}

\date{\today}
\title{Julia sets and wild Cantor sets}
\author{Alastair Fletcher}
\address{Department of Mathematical Sciences, Northern Illinois University, DeKalb, IL 60115-2888. USA}
\email{fletcher@math.niu.edu}

\author{Jang-Mei Wu}
\address{Department of Mathematics, University of Illinois,  1409 West Green Street, Urbana-Champaign, IL 61822, USA}
\email{wu@math.uiuc.edu}
\thanks{ J.-M.W. is partially supported by the National Science Foundation grant DMS-1001669.}


\begin{abstract}
There exist uniformly quasiregular maps $f:\R^3 \to \R^3$ whose Julia sets are wild Cantor sets.
\end{abstract}

\maketitle

\section{Introduction}\label{sec:introduction}

The most direct generalization of the iteration of holomorphic functions in the plane to higher real dimensions is the iteration of uniformly quasiregular mappings in $\R^n$. Informally speaking, quasiregular mappings allow a bounded amount of distortion. Uniformly quasiregular mappings, abbreviated to uqr mappings, have a uniform bound for the distortion of all the iterates. These were introduced by Iwaniec and Martin \cite{IM} who showed that there are direct analogues of the Julia set and Fatou set for uqr mappings and that $\R^n = J(f) \cup F(f)$. It is known that certain structures can arise as Julia sets of uqr mappings:
\begin{itemize}
\item all of $\R^n$, arising from Latt\`es type examples \cite{May},
\item the unit sphere $S^{n-1} \subset \R^n$ arising from uqr versions of power mappings \cite{May},
\item an $(n-1)$-ball $B^{n-1} \times \{ 0 \}$ in $\R^n$ arising from uqr versions of Chebyshev polynomials \cite{May},
\item tame Cantor sets arising from conformal trap methods \cite{IM,MP}.
\end{itemize}
Here, a tame Cantor set $E\subset \R^n$ is one for which there exists a homeomorphism $\psi :\R^n \to \R^n$ such that $\psi(E)$ is a usual ternary Cantor set contained in a line. A Cantor set which is not tame is called wild. The first example of a wild Cantor set was Antoine's necklace \cite{Ant}, the construction of which we recall below. In the context of quasiregular mappings, Heinonen and Rickman \cite{HeiRic} constructed quasiregular mappings in $\R^3$ that branch on a wild Cantor set. In this paper, we prove the following theorem.

\begin{theorem}\label{thm:main}
Let $m\in \N$ be the square of a sufficiently large even integer. Then there exist an Antoine's necklace $X\subset \R^3$ of multiplicity $m$ and a uniformly quasiregular map $f:\R^3 \to \R^3$ whose Julia set $J(f)$ is $ X$. Further, $J(f)$ is the closure of the repelling periodic points of $f$.
\end{theorem}

Antoine's necklace is constructed by an iterative procedure involving a link of $m$ solid tori contained in a solid torus. This even integer $m$ is called the multiplicity of the necklace. If $m$ is sufficiently large, then the necklace can be taken to be geometrically self-similar.
The idea in the proof of this theorem is to interpolate between conformal similarities arising from the necklace construction and the uqr power mapping of Mayer \cite{May} by first constructing an explicit branched cover, then applying an extension theorem of Berstein and Edmonds \cite{BerEdm}.

The construction here cannot happen in the plane since there are no wild Cantor sets in  $\R^2$.
Blankinship \cite{Bla} extended Antoine's construction to produce wild Cantor sets in higher dimensions.
It is unknown whether a basic branched cover analogous to that in Section \ref{sec:basic-cover} exists in higher dimensions and whether Berstein and Edmonds' theorem can be generalized to dimension four or higher.
For these reasons, we restrict to dimension three in this paper.

It is known that the Julia set of a uqr mapping is the closure of the periodic points \cite{Siebert}, but it is still an open question whether it is the closure of the repelling periodic points, as it is in the holomorphic setting. In \cite{Fle}, it is shown that this question has an affirmative answer when $J(f)$ is a tame Cantor set, and the final assertion of Theorem \ref{thm:main} shows that there is an affirmative answer for at least some wild Cantor sets.

Tame Cantor sets arising as Julia sets of uqr mappings were used in \cite{Fle} to construct quasiregular versions of Poincar\'e linearizers where the fast escaping set forms a structure called a spider's web. The natural question of whether there exist uqr mappings with wild Cantor sets arose from this work.

The paper is organized as follows. In section \ref{sec:pre} we recall the definition of quasiregular mappings and state results we will need for our construction. In section \ref{sec:Cantor}, the construction of Antoine's necklace is recalled and, in particular, a geometrically self-similar version. The construction and verification of the properties of the uqr map constructed for the proof of Theorem \ref{thm:main} is contained in sections \ref{sec:basic-cover} and \ref{sec:f}.
\medskip

\noindent{\bf{Acknowledgements}.} {The authors would like to thank Pekka Pankka and Kai Rajala for their helpful remarks on the manuscript, and Julie Kaufman for drafting the figures.}

\section{Preliminaries}\label{sec:pre}

We denote by $B(x,r)$ the Euclidean ball centered at $x\in \R^n$ of radius $r>0$ and by $S(x,r)$ the boundary of $B(x,r)$.

\subsection{Quasiregular maps}

A mapping $f:E \rightarrow \R^{n}$ defined on a domain $E \subseteq \R^{n}$ is called quasiregular if $f$ belongs to the Sobolev space $W^{1}_{n, loc}(E)$ and there exists $K \in [1, \infty)$ such that
\begin{equation} \label{eq2.1}
\av f'(x) \av ^{n} \leq K J_{f}(x)
\end{equation}
almost everywhere in $E$. Here $J_{f}(x)$ denotes the Jacobian determinant of $f$ at $x \in E$.
Informally, a quasiregular mapping sends infinitesimal spheres to infinitesimal ellipsoids with bounded eccentricity. We refer to Rickman's monograph \cite{Ric} for more details on quasiregular mappings.

A mapping $f:E \rightarrow \R^{n}$ defined on a domain $E \subseteq \R^{n}$ is said to be of bounded length distortion (BLD) if $f$ is sense-preserving, discrete, open and satisfies
\[
\ell(\gamma)/L \leq \ell(f\circ \gamma) \leq L\, \ell(\gamma)
\]
for some $L\ge 1$ and all paths $\gamma$ in $E$, where $\ell(\cdot)$ denotes the length of a path. BLD maps were introduced by  Martio and V\"ais\"al\"a \cite{MarVai}; see also \cite{HeiRic-BLD}. They
form a strict subclass of quasiregular maps.

\subsection{Uqr mappings}

The composition of two quasiregular mappings is again a quasiregular mapping, but the dilatation typically increases.
A quasiregular mapping $f$ is called uniformly quasiregular, or uqr, if \eqref{eq2.1} holds uniformly in $K$ over all iterates of $f$. If $f:\R^n \to \R^n$ is uqr, then the Fatou set is
\[F(f) = \{ x\in \R^n : \text{ there is a neighborhood } U \ni x \text{ such that } (f^m |_U)_{m=1}^{\infty} \text{ forms a normal family} \},\]
and the Julia set $J(f) = \R^n \setminus F(f)$, see \cite{IM}. The escaping set of a quasiregular mapping is
\[ I(f) = \{ x\in \R^n : |f^n(x)| \to \infty \}.\]
The following result is contained in \cite{FleNic}, and characterizes the Julia set of uqr mappings in terms of the escaping set:

\begin{theorem}[Lemma 5.2, \cite{FleNic}]
\label{thm:fn}
Let $f:\R^n \to \R^n$ be uqr. Then $J(f) = \partial I(f)$.
\end{theorem}

\subsection{Uqr power mappings}

There are higher dimensional uniformly quasiregular counterparts to power mappings, constructed by Mayer \cite{May}.

\begin{theorem}[Theorem 2, \cite{May}]
\label{thm:m}
For every $d\in \N$ with $d>1$, there is a uqr map $g:\overline{\R^3} \to \overline{\R^3}$ of degree $d^2$, with Julia set $J(g) = S(0,1)$ and whose Fatou set consists of $B(0,1)$ and $\overline{\R^3} \setminus \overline{B(0,1)}$.
\end{theorem}

In particular, for any $r>0$,
\begin{equation}
\label{eq:mayer}
g(B(0,r)) = B(0,r^d).
\end{equation}

\subsection{Extending branched coverings from the boundary of $3$-manifolds}

We will need the following result of Berstein and Edmonds \cite{BerEdm} on extending branched coverings over PL cobordisms.

\begin{theorem}[Theorem 6.2, \cite{BerEdm}]
\label{thm:be}
Let $W$ be a connected, compact, oriented PL $3$-manifold in some $\R^m$ whose boundary $\partial W$ consists of two components $M_0$ and $M_1$ with the induced orientation. Let $W'=N\setminus (\inte
B_0 \cup \inte B_1)$ be an oriented PL $3$-sphere $N$ in $\R^4$ with two disjoint closed polyhedral $3$-balls removed, and have the induced orientation on its boundary.
Suppose that $\varphi_i :M_i^2 \to \partial B_i$  is a sense-preserving oriented branched covering of degree $n\geq 3$, for each $i=0,1$. Then there exists  a sense-preserving PL branched cover $\varphi : W \to W'$ of degree $n$ that extends $\varphi_0$ and $\varphi_1$.

\end{theorem}

This theorem of Berstein and Edmonds has been generalized to branched covers $\partial W \to \partial W'$ between boundaries of connected, compact, oriented $3$-manifolds $W$ and $W'$, when $\partial W$ has  $p\ge 2$ connected components and $W'$ is a PL $3$-sphere with the interiors of $p$ disjoint closed $3$-balls removed. This generalization was proved first by Heinonen and Rickman \cite{HeiRic-BLD} for mappings whose  degrees are large multiples of $3$, and later in \cite{PRW} for all degrees $\ge 3$. This theorem is false when the degree is $2$ by an example of Fox, see \cite{BerEdm-degree}.

When the components of $\partial W'$ of the target manifold $W'$ have varying topological types, there seems to be no general procedure for extending branched covers from the boundary to the inside.

\section{Wild Cantor sets}\label{sec:Cantor}

We briefly recall the construction of the wild Cantor set in $\R^3$ constructed by Antoine in 1921 \cite{Ant};
see the book of Rolfsen (\cite{Rolfsen}, p. 73) for an illustrated description.

\subsection{Antoine's necklace}
Let $X_0 \subset \R^3$ be a solid torus  and let $m\ge 4$ be a positive even integer. Choose mutually disjoint solid tori $X_{1,1},\ldots, X_{1,m}$ contained in the interior of $X_0$ so that  $X_{1,i}$ and $X_{1,j}$ are linked if and only if $|i-j| \equiv 1 (\text{mod}\,m)$ and, when linked, they form a Hopf link. Fix homeomorphisms $\varphi_j : X_0 \to X_{1,j}$ for $j=1,\ldots,m$, and
define
\[ X_1 = \bigcup_{j=1}^m X_{1,j} = \bigcup_{j=1}^m \varphi_i(X_0).\]
We then inductively define
\[ X_{k+1} = \bigcup_{j=1}^m \varphi_j(X_k),\]
for $k\geq 1$. At the $k$'th stage $X_k$ will consist of $m^k$ tori $X_{k,1}, \ldots, X_{k,m^k}$. Clearly, $X_{k+1} \subset X_k$.
An Antoine's necklace of multiplicity $m$ is defined as
\[ X = \bigcap_{k=1}^{\infty} X_k.\]
If $c_k$ is the maximum diameter of any torus in $X_k$, then we require that $c_k \to 0$ as $k\to \infty$ in order to obtain that $X$ is a Cantor set.
The set $X$ is topological self-similar. However, for our purposes we need a geometrically self-similar necklace.

\subsection{A geometrically self-similar necklace}\label{sec:necklace-geom}

Let $m$ be a large even integer,  $p_1,\ldots, p_m$ be $m$ equally spaced points on the unit circle
$\tau_0= \{(x_1,x_2,0) \colon x_1^2 + x_2^2=1\}$, and $T_0$ be the solid torus $\{x\in \R^3\colon \dist(x,\tau_0) \leq 8/m \}$ with core $\tau_0$.
When $m$ is sufficiently large, circles $\tau_j, j=1,\ldots, m,$ in $\R^3$ with centers $p_j$ and radius $4/m$ may be fixed so that
\begin{itemize}
\item $\tau_i$ and $\tau_j$ are linked in $\R^3$ if and only if $|i-j| \equiv 1 (\text{mod}\,m)$;
\item $\rho (\tau_j)=\tau_{j+2}$ for $ j=1,\ldots, m-2,$ $\rho (\tau_{m-1})=\tau_1$, and $\rho(\tau_m)= \tau_2$, where $\rho$ is the rotation about  the $x_3$-axis by an angle $4\pi/m$,
\[
\rho\colon (r,\theta,x_3)\mapsto (r,\theta+\frac{4\pi}{m},x_3);
\]
\item $\tau_1$ and $\tau_m$ are linked with the $x_1$ axis as shown in Figure \ref{Hopflink5}, and a rotation through angle $\pi$ about the $x_1$ axis sends $\tau_1$ onto $\tau_m$ and vice versa.
\end{itemize}

Fix $m$ sense preserving similarities $\varphi_j,  j=1,\ldots, m,$ of $\R^3$  with
$\varphi_j(\tau_0)=\tau_j$ and set $T_j=\varphi_j (T_0)$. Then  $T_j$, for $j=1,\ldots, m,$ are mutually disjoint tori in the interior of $ T_0$. A geometric self-similar necklace $X$ may be obtained by setting
$X_0=T_0$ and $X_{1,j}= T_j$ for $j=1,\ldots,m$ in the above construction.

\begin{figure}[h!]
\begin{center}
\includegraphics[scale=1.0]{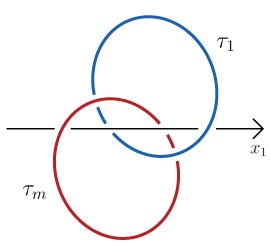}
\caption{}
\label{Hopflink5}
\end{center}
\end{figure}

\section{A basic covering map}\label{sec:basic-cover}

Towards the proof of Theorem \ref{thm:main}, we construct a BLD degree $m$ branched covering map
\[F \colon\, T_0 \setminus \inte \left ( \bigcup_{j=1}^m\, T_j \right )\,  \longrightarrow \, \overline{B(0,2)} \setminus \inte (T_0)\]
satisfying $F|\partial T_j\colon \partial T_j \to \partial T_0 = \varphi_j^{-1}$ for the tori $T_0, T_1,\ldots, T_m$ fixed
 in Section \ref{sec:necklace-geom}.

\begin{figure}[h!]
\begin{center}
\includegraphics[scale=1.0]{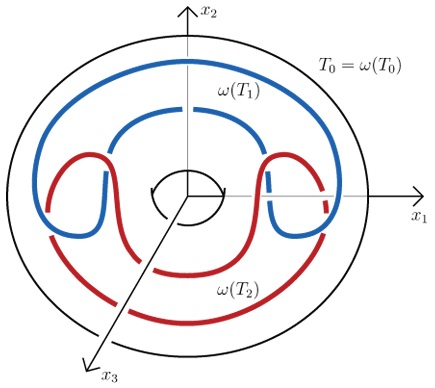}
\caption{}
\label{Bingdoublinglink}
\end{center}
\end{figure}

Let $\omega \colon \R^3 \to \R^3$ be the degree $m/2$ winding map
\[
\omega \colon (r,\theta, x_3)\mapsto (r, \theta m/2, x_3).
\]
Then $\omega\colon T_0\to T_0$ is an unbranched cover that maps all $T_j$ with odd indices to $\omega(T_1)$ and all $T_j$ with even indices to $\omega (T_2)$. Moreover $\omega(T_1)$ and $\omega (T_2)$ are linked inside $T_0$ as in Figure \ref{Bingdoublinglink} and by construction are symmetric under the involution $\iota$ that corresponds to the rotation about the $x_1$-axis by an angle $\pi$
\[
\iota \colon (x_1,x_2,x_3) \mapsto (x_1, -x_2,-x_3).
\]
That is, we have $\iota (\omega (T_1))=\omega (T_2)$ and $\iota (\omega (T_2))=\omega (T_1)$.

The quotient  $q \colon T_0 \to T_0/\langle \iota \rangle$ is a degree 2 sense preserving map, under which $q(\omega (T_1))=q(\omega (T_2))$
 is a torus unknotted  in the $3$-cell $q(T_0)$;
see for example \cite[p. 294]{Rolfsen}.
To obtain a BLD branched cover from  $T_0$ onto $\overline{B(0,2)}$ which unlinks the tori $\omega (T_1)\cup \omega (T_2)$, we consider a PL version of $q$. Give  $T_0$ a $C^1$-triangulation $g\colon |U| \to T_0$  in the sense of \cite[p. 81]{Munkres} by a simplicial complex $U$ in $\R^3$, that respects the involution $\iota|T_0$ and of which $g^{-1}(\omega(T_1) \cup \omega(T_2))$ is a subcomplex.
Identify $q(T_0)$ with a simplicial complex $V$ in $\R^3$ of which $q(\omega (T_1))$ is a subcomplex and that the quotient  $q \circ g\colon U \to q(T_0)$ is simplicial. Under these identifications, $q$ is a BLD map.

We may assume, after a simplicial subdivision of  $V$,  that $\overline{B(0,2)}$ has a $C^1$-triangulation $h\colon |V| \to \overline{B(0,2)}$ under which  $h^{-1}(T_0)$ is  a subcomplex and  the map $ q \circ \omega \circ \varphi_1 \circ h|h^{-1}(T_0) \colon h^{-1}(T_0) \to q(\omega(T_1))$ is simplicial. Since $q \circ \omega \circ \varphi_1|T_0$ and  $h^{-1}|T_0$ are ambient isotopic in $V$, there exists  a homeomorphism $\zeta \colon \overline{B(0,2)} \to V $ such that $ \zeta|T_0 = q \circ \omega \circ \varphi_1 $ and $\zeta=h^{-1}$ on $\partial B(0,2)$, and that $\zeta^{-1} \circ h \colon V \to V $ is PL by an isotopy extension theorem \cite[p. 136]{Hudson}.

It follows that
$\zeta\circ q \circ \omega|T_j=\varphi_j^{-1}, j=1, \ldots, m,$ are similarities.
Then $F=\zeta \circ q \circ \omega$ is the BLD degree $m$ branched covering map in the claim.

\section{Proof of Theorem \ref{thm:main}}\label{sec:f}

To construct the uqr map $f$ for Theorem \ref{thm:main}, we require $m =d^2$ to be the square of an even integer sufficiently large so that the construction of Antoine's necklace allows conformal similarities $\varphi_1,\ldots,\varphi_m$. Let $X_0=T_0$, $X_{1,j}= T_j$ for $j=1,\ldots,m$  and $X$ the geometrically self-similar necklace of multiplicity $m$ from Section \ref{sec:necklace-geom}.

Set $B_0=B(0,2)$, $B_{-1}=B(0, 2^d)$ and write
$\R^3$ as  disjoint unions, one for the domain and one for the target, as follows:
\[
\R^3= X_1 \cup (X_0\setminus X_1) \cup (B_0\setminus X_0) \cup (\R^3\setminus B_0)
\]
and
\[
\R^3= X_0 \cup (B_0\setminus X_0) \cup (B_{-1}\setminus B_0) \cup (\R^3\setminus B_{-1}).
\]
The uqr map $f$ will be defined in four matching parts.

\begin{enumerate}[(i)]

\item Set $f \colon \overline{X_0\setminus X_1} \to \overline{B_0\setminus X_0}$  to be the degree $m$ basic branched covering map $F$ from Section \ref{sec:basic-cover}.

\item Extend $F$ to $X_1$ by defining  $f|X_1$ to be $\varphi_j^{-1}:X_{1,j} \to X_0$  for every $j=1,\ldots,m$.

\item Define $f\colon \R^3\setminus \text{int}( B_0) \to \R^3\setminus \text{int}(B_{-1})$ to be the restriction of $g$ to $\R^3\setminus \text{int}( B_0)$, recalling the uqr power mapping $g$ of degree $m$ from Theorem \ref{thm:m}. By \eqref{eq:mayer}, $S(0,2)$ is mapped onto $S(0,2^d)$ by $g$. We remark also that $g|_{S(0,2)}$ is orientation preserving by \cite[Proposition 5.2]{May}.

\item Since $f|\partial B_0\colon \partial B_0\to \partial B_{-1}$ is a BLD degree $m$  branched cover and $f|\partial X_0\colon \partial X_0\to \partial B_0$ is an $m$-fold cover by similarities, we may extend the boundary map to a BLD degree $m$
    branched cover $f\colon B_0 \setminus \text{int} X_0 \to B_{-1}\setminus \text{int} B_0$
by Theorem \ref{thm:be}, the extension theorem of Berstein and Edmonds. It is understood here that
$C^1$-triangulation has been carried out on $B_0 \setminus \text{int} X_0$ and $ B_{-1}\setminus \text{int} B_0$ before applying Theorem \ref{thm:be}.

\end{enumerate}

This completes the construction of a quasiregular map $f\colon \R^3 \to \R^3$. To finish the proof of Theorem \ref{thm:main}, we show that $f$ has the required properties in the following lemmas.

\begin{lemma}
The map $f$ is a uniformly quasiregular mapping of polynomial type.
\end{lemma}

\begin{proof}
Let $x\in \R^3$. If the orbit of $x$ always remains in $X_1$, then since $f|_{X_1}$ is conformal the dilatation of $f^k$ at $x$ will always be $1$.

Otherwise, after finitely many iterations through at worst conformal maps and two quasiregular maps, $f^{k_0}(x)\in \R^n \setminus B(0,2^d)$. From this point on, $f$ agrees with the uqr power map of degree $m$ and hence the dilatation will remain bounded. In short, along the orbit of a point which starts in $X_0$ but eventually leaves $X_0$, $f$ consists of finitely many conformal maps, two quasiregular mappings and then a uqr mapping.
Hence the dilatation of $f^k$ remains uniformly bounded on $\R^3$ as $k\to \infty$.

Since $f$ has finite degree $m$, it is of polynomial type.
\end{proof}

\begin{lemma}
The Julia set of $f$ is equal to $X$.
\end{lemma}

\begin{proof}
By the construction of $f$, if the orbit of any point leaves $X_0$, then it is contained in the escaping set $I(f)$ of $f$. Also by construction, the set of points which do not leave $X_0$ under iteration of $f$ are exactly the points in $X$. If $x\in X$, then any sufficiently small neighbourhood of $x$ will intersect the boundary of $X_k$ for some $k$. Since by construction $\partial X_k \subset I(f)$,
we obtain $X = \partial I(f)$. By Theorem \ref{thm:fn}, $\partial I(f) = J(f)$ and hence $J(f) =X$.
\end{proof}

\begin{lemma}
The Julia set $J(f)$ is the closure of the repelling periodic points.
\end{lemma}

\begin{proof}
By a result in Siebert's thesis \cite{Siebert}, see also \cite[Theorem 4.1]{Bergweiler} and the discussion preceding it, the periodic points are dense in $J(f)$. Let $x_0$ be a periodic point in $J(f)$ of period $p$. Then, recalling the construction of the Antoine's necklace $X$, there exist integers $j,k$ such that
\[ x_0 \in \operatorname{int}(X_{1,j}) \subset X_1 \quad \text{ and } \quad x_0 \in \operatorname{int}(X_{p+1,k}) \subset X_{p+1} \cap X_{1,j}.\]
In particular, $X_{p+1,k} \subset \operatorname{int}(X_{1,j})$. By the construction of $f$, $f^p:X_{p+1,k} \to X_{1,j}$ is injective. Hence by the topological definition of fixed points, see \cite[p.90]{HMM}, $x_0$ is a repelling fixed point of $f^p$. Finally, by \cite[Proposition 4.6]{HMM}, $x_0$ is a repelling periodic point of $f$.
\end{proof}

This completes the proof of Theorem \ref{thm:main}.

\end{document}